\documentclass[a4paper]{amsart}

\usepackage[T1]{fontenc}
\usepackage[utf8]{inputenc}
\usepackage{paralist}
\usepackage{microtype}
\usepackage{lmodern}
\usepackage{dsfont,mathrsfs}
\usepackage{amsmath,amsthm,amssymb}

\theoremstyle{definition}
\newtheorem*{thm*}{Theorem}
\newtheorem{theorem}{Theorem}[section]
\newtheorem{lemma}[theorem]{Lemma}
\newtheorem{definition}[theorem]{Definition}
\newtheorem{prop}[theorem]{Proposition}
\newtheorem{cor}[theorem]{Corollary}
\newtheorem{rem}[theorem]{Remark}
\numberwithin{equation}{section}

\newcommand{\coloneqq}{\mathrel{\mathop:}=}
\newcommand{\dx}{\,\mathrm{d}}
\newcommand{\N}{\mathds{N}}
\newcommand{\eps}{\varepsilon}
\DeclareMathOperator{\fix}{Fix}
\newcommand{\applied}[2]{\langle #1,#2\rangle}
\DeclareMathOperator{\lin}{span}
\DeclareMathOperator{\Real}{Re}

\begin{document}
\title{A new proof of Doob's theorem}
\author{Moritz Gerlach}
\address{Moritz Gerlach\\University of Ulm\\Institute of Applied Analysis\\89069 Ulm\\Germany}
\email{moritz.gerlach@uni-ulm.de}
\author{Robin Nittka}
\address{Robin Nittka\\University of Ulm\\Institute of Applied Analysis\\89069 Ulm\\Germany}
\email{robin.nittka@uni-ulm.de}
\subjclass[2010]{Primary 37A10; Secondary 47G10, 47D07}
\keywords{Markovian semigroups, strong Feller, asymptotics, integral operator}
\date{}
\begin{abstract}
We prove that every bounded, positive, irreducible, stochastically
continuous semigroup on the space of bounded, measurable functions
which is strong Feller, consists of kernel operators and possesses an
invariant measure converges pointwise. This differs from Doob's
theorem in that we do not require the semigroup to be Markovian and
request a fairly weak kind of irreducibility. In addition, we
elaborate on the various notions of kernel operators in this context,
show the stronger result that the adjoint semigroup converges strongly
and discuss as an example
diffusion equations on rough domains. The
proofs are based on the theory of positive semigroups and do not use
probability theory.
\end{abstract}
\maketitle

\section{Introduction}

One of Doob's celebrated theorems states that, given an irreducible, strong
Feller, stochastically continuous, Markovian transition semigroup $(T(t))_{t
\ge 0}$ on the space $B_b(\Omega)$ of bounded, measurable functions, where
$\Omega$ is a Polish space, with invariant probability measure $\mu$, the
transition probabilities converge to $\mu$~\cite{doob48}, see also~\cite[\S
4.2]{daprato1996}. This implies that the invariant measure is unique.
The conclusion of Doob's theorem can be formulated equivalently by saying that
$T(t)f \to \int_\Omega f \; \dx\mu$ pointwise on $\Omega$ for all $f \in
B_b(\Omega)$ as $t \to \infty$. Later, this result has been strengthened,
proving that in fact $T(t)' \nu \to \mu$ for any probability measure $\nu$ on $\Omega$ in the
total variation norm~\cite{seidler1997,stettner1994}. Applications of this
theorem include elliptic equations with unbounded coefficients
in $\mathds{R}^N$~\cite{metafune2002}, e.g.\ Ornstein-Uhlenbeck semigroups~\cite{daprato1995},
and infinite-dimensional diffusion~\cite{peszat1995}.

Since the original motivation for investigating such convergence behavior and most
examples to which Doob's theorem applies stem from probability theory, 
it is only natural that the known proofs are of probabilistic nature.
We consider it worthwhile to take a look at the result from a more
abstract viewpoint, which is strongly influenced by the theory of positive
semigroups. Doob's theorem can be obtained rather easily from a
very general result due to Greiner~\cite{greiner1982} about convergence of
positive semigroups on $L^p$-spaces, see Section~\ref{sec:convergence}. We 
regard this as interesting by itself.
Moreover, by using this different approach we are able to generalize
Doob's theorem. On the one hand, we can relax the usual irreducibility condition. On
the other hand, and maybe more importantly, we do no longer require the semigroup
to be Markovian, but only assume that the semigroup is positive and 
bounded. See Theorem~\ref{thm:pointwiseconvergence} for our main result.

Another motivation is to make Doob's theorem more accessible to 
people not working in probability theory, translating the
assumptions of the theorem into a more functional analytic language. For example,
even though it is natural to assume that the semigroup possesses transition
probabilities if one thinks of stochastic processes, such a condition is
unnatural and can be hard to verify for other classes of semigroups, for example
for semigroups arising from partial differential equations. Fortunately,
we are able to show in Section~\ref{sec:kernel} that for strong Feller
semigroups the existence of transition probabilities is equivalent to the
assertion that the semigroup has an integral kernel representation, the
existence of integral kernels being a well-studied topic in the theory of 
differential equations.
We show in Section~\ref{sec:examples} by a simple example how this equivalence helps to
apply Doob's theorem to diffusion semigroups.

\bigskip

Let us compare the conditions of our version of Doob's theorem with those of the classical one.
In \cite[\S 4.2]{daprato1996} the semigroup $\mathcal{T}=(T(t))_{t\geq 0}$ is assumed to be $t_0$-regular for some $t_0>0$,
i.e., for every $0\leq f\in B_b(\Omega)$
either $T(t_0)f=0$ or $(T(t_0)f)(x)>0$ for all $x\in\Omega$.
If $\mathcal{T}$ is strong Feller at $s_0 > 0$, i.e., $T(s_0)$ maps $B_b(\Omega)$ into $C_b(\Omega)$, the space of
continuous bounded functions, and $\mathcal{T}$ is $t_0$-irreducible, i.e.,
$(T(t_0)\mathds{1}_B)(x)>0$ for all $x\in\Omega$ and every non-empty open $B\subset \Omega$, then
$\mathcal{T}$ is $(t_0+s_0)$-regular. Indeed, this is the most common way to check that a semigroup
is $t_0$-regular for some $t_0 > 0$.  But in fact a positive $t_0$-regular semigroup
is necessarily $t_0$-irreducible in a sense made precise in the following remark. In this sense the classical
version of Doob's theorem assumes $t_0$-irreducibility.

\begin{rem}\label{rem:regirr}
	Let $\Omega$ be a Polish space. Then every positive $t_0$-regular semigroup with invariant 
	measure $\mu$ is $t_0$-ir\-re\-du\-ci\-ble up to neglecting a $\mu$-nullset. In fact,
	let $\mathcal{T}$ be $t_0$-regular and denote by $U$ the union of all open $\mu$-nullsets. Then
	\[ \int_\Omega T(t_0)\mathds{1}_U \dx\mu = \int_\Omega \mathds{1}_U \dx\mu = \mu(U) = 0\]
	shows that $T(t)\mathds{1}_U = T(t-t_0)T(t_0)\mathds{1}_U=0$ for all $t\geq t_0$. Thus
	$T(t)f = T(t)(f\mathds{1}_{U^c})$ for all $f \in B_b(\Omega)$.
	Now if $B$ is a non-empty relatively open subset of $\Omega \setminus U$, then
	\[ \int_\Omega T(t_0) \mathds{1}_B \dx\mu = \int_\Omega \mathds{1}_B \dx\mu = \mu(B) > 0,\]
	hence $(T(t_0)\mathds{1}_B)(x) > 0$ for all $x \in \Omega$ since $\mathcal{T}$ is $t_0$-regular.
	This proves that neglecting the nullset $U$, we obtain a $t_0$-irreducible semigroup.
\end{rem}

We show that for strong Feller Markov semigroups we can drop the assumption of $\mathcal{T}$ being $t_0$-irreducible
if we already know that the invariant measure $\mu$ is strictly positive, i.e., $\mu(B) > 0$ for every
non-empty open $B \subset \Omega$, and if we require $\Omega$ to be connected,
see Theorem~\ref{thm:pointwiseconvergence} and Proposition~\ref{prop:TL1irred}.
If we do not assume the semigroup to have a fixed point, the invariant measure to be strictly positive, 
or if we allow $\Omega$ to be disconnected, we have to assume some kind of irreducibility, namely 
that there is no closed set $\emptyset\neq A \subsetneq \Omega$ such that the closed
ideal $\{ f\in B_b(\Omega) : f_{\mid A}=0 \}$ of $B_b(\Omega)$ is invariant under the action of $\mathcal{T}$.
This notion of irreducibility is weaker than $t_0$-irreducibility.


Moreover, we show that the adjoint semigroup converges strongly on the space of measures, i.e., 
\[\lim_{t\to\infty} \Vert T'(t) \nu  - \applied{e}{\nu} \mu \Vert_{TV} = 0\]
for every finite Borel measure $\nu$, see Theorem~\ref{thm:adjointconv}.
This is the generalization of the above mentioned result by Seidler~\cite[Prop~2.5]{seidler1997} and 
Stettner~\cite[Thm~1]{stettner1994} to our setting.

Before proving our main theorem in Section~\ref{sec:convergence}, in Section~\ref{sec:kernel} we discuss
the assumption that $T'(t)\delta_x$ be a measure for all $x \in \Omega$.
Operators possessing this property are frequently called
kernel operators, whereas in some other areas of mathematics, particularly in the theory of partial differential equations,
the term \emph{kernel operator} is reserved for the class
of operators admitting an integral representation of the form
\[ (Tf)(x) = \int_\Omega k(x,y)f(y) \dx\mu(y) \quad (x\in\Omega,\, f\in B_b(\Omega))\]
for a jointly measurable function $k:\Omega\times\Omega \to \mathds{R}$, where the
equality is understood to hold for almost every $x \in \Omega$.
We show that for strong Feller semigroups admitting a strictly positive invariant measure
these two notions coincide, see Theorem~\ref{thm:ultrafeller}, so in the context of our article
there is no danger of confusion regarding these different notions.
This equivalence of different notions of kernels operators might be of independent mathematical interest,
compare for example~\cite{bukhvalov1978,arendt1994} where descriptions of kernel operators are given.
Part of this characterization is also needed for the proof of our main result.

\section{Notations}

Let us first fix some notation that will be used throughout.
By $\Omega$ we always denote a Hausdorff topological space.
We write $B_b(\Omega)$ for the space of bounded Borel measurable functions on $\Omega$
and $C_b(\Omega)$ for the space of bounded continuous functions on $\Omega$.

\begin{definition}
Let $\mu$ be a positive, finite Borel measure on $\Omega$.
We say that $\mu$ is \emph{strictly positive} if $\mu(U)>0$ for every non-empty open set $U\subset\Omega$.
If $T$ is a bounded linear operator on $B_b(\Omega)$ such that
\begin{align}
	\int_\Omega Tf \dx\mu = \int_\Omega f\dx\mu \quad (f\in B_b(\Omega)), \label{eqn:muinvariant}
\end{align}
i.e., $T'\mu=\mu$, the measure $\mu$ is called \emph{$T$-invariant}.
A bounded linear operator $T$ on $B_b(\Omega)$ that satisfies $TB_b(\Omega) \subset C_b(\Omega)$ will be called a \emph{strong Feller operator}.
\end{definition}

\begin{rem}
\label{rem:extension}
	Let $T$ be a positive (and hence bounded) linear operator on $B_b(\Omega)$ and $\mu\neq 0$ a positive, 
	finite, $T$-invariant Borel measure on $\Omega$.
	Then $Tf=0$ almost everywhere whenever $f$ vanishes almost everywhere.
	Thus, $T$ induces a bounded operator on the dense subspace $L^\infty(\Omega,\mu)$ of $L^1(\Omega,\mu)$, and
	$\Vert Tf \Vert_{L^1} \leq \Vert f\Vert_{L^1}$ for all $f\in L^\infty(\Omega,\mu)$.
	Hence $T$ induces a linear contraction on $L^1(\Omega,\mu)$ and by interpolation a bounded
	linear operator on $L^p(\Omega,\mu)$ for all $1< p <\infty$.
	By continuity, equation~\eqref{eqn:muinvariant} remains valid for all $f \in L^1(\Omega,\mu)$.
\end{rem}

\begin{definition}
	A family $\mathcal{T}=(T(t))_{t\geq 0}$ of bounded linear operators on a Banach space $X$ is called a
	\emph{semigroup on $X$} if $T(0)=I$ and $T(s+t) = T(s)T(t)$ for all $t,\,s \geq 0$.
	We denote by
	\[ \fix(\mathcal{T}) \coloneqq \{ x\in X  : T(t)x = x \text{ for all } t\geq 0\} \]
	the fixed space of a semigroup $\mathcal{T}$.

	A semigroup $\mathcal{T}=(T(t))_{t\geq 0}$ on $B_b(\Omega)$ is called a \emph{strong Feller semigroup} 
	if $T(t)$ is a strong Feller operator for every $t>0$.
	A positive, finite Borel measure $\mu$ on $\Omega$ is said to be $\mathcal{T}$-invariant 
	if it is $T(t)$-invariant for every $t\geq 0$.
\end{definition}

\begin{definition}
	If $\mu$ is a Borel measure on $\Omega$,
	then a positive semigroup $\mathcal{T}$ on $L^p(\Omega,\mu)$, $1\leq p <\infty$, is called \emph{irreducible}
	if the only closed $\mathcal{T}$-invariant ideals $J$ in $L^p(\Omega,\mu)$ are $J = \{0\}$
	and $J = L^p(\Omega,\mu)$. For the definition of ideals, we refer to~\cite[\S 8.7]{aliprantis2006}.
\end{definition}

The notion of an irreducible semigroup on $B_b(\Omega)$ is not used in this article for the
following reason.
While every closed ideal of $L^p(\Omega,\mu)$ is of the form 
\[ \{ f\in L^p(\Omega,\mu) : f=0 \text{ on } A\}\]
for some measurable $A\subset \Omega$ whenever $\mu$ is $\sigma$-finite, see e.g. \cite[III \S1 Ex.2]{schaefer1974},
there is no convenient description of the closed ideals of $B_b(\Omega)$.
Thus, the condition that a semigroup on $B_b(\Omega)$ be irreducible is hard to check and in fact quite restrictive.
Instead, it will be sufficient to assume that there is no closed set $\emptyset \neq A \neq \Omega$ such that
the closed ideal
\[ \{ f\in B_b(\Omega) : f=0 \text{ on } A\} \]
is invariant under the action of the semigroup, which is a mild irreducibility assumption.

The following definition of a kernel operator is motivated by the use of this name in differential equations.
Even though the usage of this term varies among the different fields in mathematics,
all common notions agree for strong Feller operators at least up to taking powers of the operator,
cf.~Theorem~\ref{thm:ultrafeller}.

\begin{definition}
	\label{def:kerneloperator}
	Let $\mu$ be a positive Borel measure on $\Omega$. 
	A linear operator $T\colon X\to X$ on $X=B_b(\Omega)$ or $X=L^p(\Omega,\mu)$, $1\leq p\leq \infty$, is called a 
	\emph{kernel operator} (with respect to $\mu$) if there exists an $(\Omega\times\Omega)$-measurable real-valued
	function $k$, where we equip $\Omega \times \Omega$ with the complete product $\sigma$-algebra, such that for every $f\in X$
	we have $k(x,\cdot)f(\cdot) \in L^1(\Omega,\mu)$ and
	\[ (Tf)(x) = \int_\Omega k(x,y) f(y) \dx\mu(y) \]
	for $\mu$-almost every $x\in\Omega$.
\end{definition}

\section{Characterizations of kernel operators}
\label{sec:kernel}

In this section we compare our definition of kernel operators with related notions.
We need some preparatory results first.

\begin{lemma}
	\label{lem:kernelcontinuity}
	Let $\Omega$ be second countable, $T$ a positive strong Feller operator on $B_b(\Omega)$, and $\mu\neq 0$ a positive, finite, $T$-invariant
	Borel measure on $\Omega$. If $T$ is a kernel operator,
	then there exists a $\mu$-nullset $N\subset\Omega$ such that
	$\lim (T^2 \mathds{1}_{A_n})(x) = 0$ for all $x\in\Omega\setminus N$ and
	every sequence $(A_n)$ of Borel sets satisfying $\lim \mu(A_n)=0$.
\end{lemma}
\begin{proof}
	Let $\mathcal{B}$ be countable basis of $\Omega$ that is closed under finite unions and contains $\Omega$.
	Since $T$ is a kernel operator, there exists a $\mu$-nullset $N\subset \Omega$ such that
	\begin{align}
	 (T\mathds{1}_B)(x) = \int_\Omega k(x,y) \mathds{1}_B(y) \dx\mu(y) \quad (x\in\Omega\setminus N,\,B\in\mathcal{B}).\label{eqn:kernelbasis}
	 \end{align}

	Let $(A_n)$ be a sequence of Borel sets such that $\lim \mu(A_n)=0$ 
	and define $g_n\coloneqq T\mathds{1}_{A_n}$. Since $\mu$ is $T$-invariant, $\lim g_n =0$ in $L^1(\Omega,\mu)$ and hence $\mu$-in measure.
	Therefore, after passing to a subsequence, we may assume that
	\[ \lim_{n\to\infty} \mu ( \Omega\setminus \{ g_n \geq 1/n \}) = \lim_{n\to\infty} \mu (\{g_n< 1/n \}) = \mu(\Omega).\]
	Since $\{g_n < 1/n\}$ is open,
	for every $n\in\N$ there exists $B_n \in \mathcal{B}$ such that $B_n\subset \{ g_n < 1/n\}$
	and $\mu(B_n)\geq \mu(\{ g_n < 1/n\}) - 1/n$. Hence, $\lim \mu(B_n) = \mu(\Omega)$ and thus, after passing to a subsequence,
	$\mathds{1}_{B_n}(x) \to 1$ for almost every $x\in\Omega$.
	By \eqref{eqn:kernelbasis} this proves that $\lim (T\mathds{1}_{\Omega\setminus B_n})(x) = 0$ for all $x\in\Omega\setminus N$.
	Since
	\[ g_n \leq \tfrac{1}{n} \mathds{1}_\Omega + \Vert T\mathds{1}_\Omega\Vert_\infty \mathds{1}_{\Omega\setminus B_n} \quad (n\in\N),\]
	this proves the claim.
\end{proof}

\begin{prop}
	\label{prop:kerneltomeasures}
	Let $\Omega$ be second countable, $T$ a positive strong Feller operator,
	and $\mu$ a strictly positive, finite, $T$-invariant Borel measure on $\Omega$.
	If $T$ is a kernel operator,
	then for every $x\in\Omega$ there exists a finite positive Borel measure $\mu_x$ such that
	\[ (T^2 f)(x) = \int_\Omega f\dx\mu_x\]
	holds for all $f\in B_b(\Omega)$.
\end{prop}
\begin{proof}


	For every $x\in\Omega$ and every Borel set $A\subset \Omega$ define 
	$\mu_x(A) = (T^2\mathds{1}_A)(x)$. Then $\mu_x(\emptyset)=0$, $\mu_x(\Omega)<\infty$, and $\mu_x$ is finitely additive
	for all $x\in\Omega$.
	By Lemma~\ref{lem:kernelcontinuity}, there exists a nullset $N\subset \Omega$ such that 
	$\mu_x$ is $\sigma$-additive if $x\in\Omega\setminus N$.
	Since $\Omega\setminus N$ is dense in $\Omega$ and
	\[ \lim_{y\to x} \mu_y (A) = \lim_{y\to x}(T^2 \mathds{1}_A)(y) = (T^2 \mathds{1}_A)(x) = \mu_x(A) \]
	holds for every $x\in\Omega$ and every Borel set $A\subset \Omega$,
	$\mu_x$ is a measure for all $x\in\Omega$ by the Vitali-Hahn-Saks theorem \cite[Cor III.7.4]{dunford1958}.
	Hence
	\[ (T^2f)(x) = \int_\Omega f\dx\mu_x \quad (x\in\Omega)\]
	holds for every simple function $f\in B_b(\Omega)$ and thus for every $f \in B_b(\Omega)$.
\end{proof}

The idea of the following proposition is essentially taken from~\cite{hairer09}.
Nevertheless, we include a proof here in order to be more self-contained.

\begin{prop}
	\label{prop:densitytoequicontinuity}
		Let $\Omega$ be second countable,
		$T$ a positive strong Feller operator, and $\mu\neq 0$ a positive, finite, $T$-invariant Borel measure such that
		for every $x\in\Omega$ there exists $k_x \in L^1(\Omega,\mu)$ satisfying 
		\[ (Tf)(x) = \int_\Omega k_x f\dx\mu \quad (f\in B_b(\Omega)).\]
		Then for each $c \ge 0$ the family $\{ T^2 f : f\in B_b(\Omega), \vert f\vert\leq c \}$ is equicontinuous.
\end{prop}
\begin{proof}
	Let $x\in\Omega$, $c \ge 0$, and $B\coloneqq \{ f\in B_b(\Omega) : \vert f \vert\leq c \}$.
	Assume that $\{ T^2 f: f\in B\}$ is not equicontinuous at $x$. Then there exists $\eps_0>0$ as well as sequences $(f_n)\subset B$
	and $(x_n) \subset \Omega$ such that $\lim x_n = x$ and 
	\[ \vert (T^2 f_n)(x_n)  - (T^2 f_n)(x) \vert \geq \eps_0 \quad (n\in\N).\]
	Since $L^1(\Omega,\mu)$ is separable~\cite[365Xp]{fremlin2004} and the sequence $(f_n)$ is uniformly bounded,
	there exists $f\in B$ such that 
	\[ \lim_{n\to\infty} \int_\Omega f_n g\;\dx\mu  = \int_\Omega f g\;\dx\mu \quad (g\in L^1(\Omega,\mu))\]
	after passing to a subsequence. In particular,
	\begin{align}
	 \lim_{n\to\infty} (Tf_n)(z) = \lim_{n\to\infty} \int_\Omega k_z f_n \dx\mu = \int_\Omega k_z f \dx\mu 
	 = (Tf)(z) \quad (z\in\Omega).\label{eqn:T(tau)g_n}
	\end{align}

	For $n\in\N$ and $z\in\Omega$ define $h_n(z) \coloneqq \sup_{m\geq n} \vert (T(f_m-f))(z)\vert$. Then $(h_n)$ is 
	a decreasing sequence of positive $\mu$-measurable functions such that
	$h_n \leq 2c \; T\mathds{1}_\Omega$.
	Moreover, $\lim_{n\to\infty} h_n(z)=0$ for all $z\in\Omega$ by \eqref{eqn:T(tau)g_n}.
	Now the dominated convergence theorem yields that $\lim (Th_n)(z) = 0$ for all $z\in\Omega$. Since
	\[ \limsup_{n\to\infty} (Th_n)(x_n) \leq \limsup_{n\to\infty} (Th_m)(x_n) = (Th_m)(x) \quad (m\in\N),\]
	this implies that $\lim_{n\to\infty} (Th_n)(x_n)=0$.
	Thus
	\begin{align*}
		& \limsup_{n\to\infty} \vert (T^2 f_n)(x_n) - (T^2f)(x) \vert \\
		& \quad \leq \limsup_{n\to\infty} \bigl(\vert (T^2f_n)(x_n) - (T^2f)(x_n) \vert + \limsup_{n\to\infty}\vert (T^2f)(x_n) - (T^2f)(x) \vert \bigr)\\
		& \quad \leq \limsup_{n\to\infty} (T h_n)(x_n) = 0.
	\end{align*}
	In addition, $\lim (T^2 f_n)(x) = (T^2 f)(x)$ by \eqref{eqn:T(tau)g_n}.
	Thus
	\[ \lim_{n\to\infty} \vert (T^2 f_n)(x_n) - (T^2 f_n)(x) \vert = 0,\]
	which proves equicontinuity of $B$.
\end{proof}

The following  characterization of kernel operators is the main result of this section.
As a consequence of this theorem, a strong Feller semigroup consists of kernel operators if and only if
it has a representation via transition probabilities.

\begin{theorem}
	\label{thm:ultrafeller}
	Let $\Omega$ be second countable, $\mathcal{T}=(T(t))_{t\geq 0}$ a 
	positive strong Feller semigroup on $B_b(\Omega)$, and $\mu$ a strictly positive, finite, $\mathcal{T}$-invariant Borel measure on $\Omega$.
	Then the following assertions are equivalent.
	\begin{enumerate}[(i)]
		\item\label{assertionkernel} For every $t > 0$, the operator $T(t)$ is a kernel operator with respect to $\mu$
		in the sense of Definition~\ref{def:kerneloperator}.
		\item\label{assertion1*}
		For all $t>0$ and every uniformly bounded sequence $(f_n)\subset B_b(\Omega)$ such that
		$\lim f_n = 0$ $\mu$-in measure holds $\lim (T(t)f_n)(x) =0$ for almost every $x\in \Omega$.
		\item\label{assertion1}
		For all $t>0$ and every uniformly bounded sequence $(f_n)\subset B_b(\Omega)$ such that
		$\lim f_n = 0$ $\mu$-in measure holds $\lim (T(t)f_n)(x) =0$ for every $x\in \Omega$.
		\item\label{assertion2} For every $x\in \Omega$ and $t>0$ there exists a positive finite Borel measure $\mu_{t,x}$ such that
		\[ (T(t)f)(x)= \int_\Omega f\dx\mu_{t,x} \quad (f\in B_b(\Omega)),\]
		i.e., $T(t)'\delta_x \in B_b(\Omega)'$ is a measure for each $x\in\Omega$ and $t>0$.
		\item\label{assertion3} For every $x\in\Omega$ and $t>0$ there exists $k_{t,x} \in L^1(\Omega,\mu)$ such that
		\[ (T(t)f)(x) = \int_\Omega k_{t,x} f\dx\mu \quad (f\in B_b(\Omega))).\]
		\item\label{assertion4} The family $\{ T(t)f : f\in B_b(\Omega), \vert f\vert\leq c \}$ is equicontinuous for all $c,t>0$.
	\end{enumerate}
\end{theorem}
\begin{proof}
	(\ref{assertionkernel}) $\Leftrightarrow$ (\ref{assertion1*}): This equivalence is due to Bukhvalov \cite[Thm~2]{bukhvalov1978},
	see also \cite[Thm 3.3.11]{meyer1991}.

	(\ref{assertionkernel}) $\Rightarrow$ (\ref{assertion2}): This follows from Proposition~\ref{prop:kerneltomeasures} and the semigroup property.

	(\ref{assertion2}) $\Rightarrow$ (\ref{assertion3}):
	By the Radon-Nikodym theorem, it suffices to show that each $\mu_{t,x}$ is absolutely continuous with respect to $\mu$.
	Let $t>0$, $x\in\Omega$, and $A\subset \Omega$ be a Borel set with $\mu(A)=0$.  Then
	\[ \int_\Omega T(t)\mathds{1}_A \dx \mu = \int_\Omega \mathds{1}_A \dx \mu = \mu(A)=0 \]
	and hence $T(t)\mathds{1}_A=0$ almost everywhere.
	Since $T(t)\mathds{1}_A$ is continuous and $\mu$ is strictly positive, we conclude that $(T(t)\mathds{1}_A)(x) = 0$
	for all $x \in \Omega$, hence
	\[ 0 = (T(t)\mathds{1}_A)(x) = \int_\Omega \mathds{1}_A \dx \mu_{t,x} = \mu_{t,x}(A) \quad(x\in\Omega).\]

	(\ref{assertion3}) $\Rightarrow$ (\ref{assertion4}): This follows from Proposition~\ref{prop:densitytoequicontinuity}
	and the semigroup property.

	(\ref{assertion4}) $\Rightarrow$ (\ref{assertion1}): Let $t>0$ and fix $(f_n)\subset B_b(\Omega)$ such that
	$\vert f_n \vert \leq c$ for some $c > 0$ and $\lim f_n = 0$ in measure.
	Then $(f_n)$ converges to zero in the norm of $L^1(\Omega,\mu)$.
	Assume that there exists $x\in\Omega$ such that $(T(t)\vert f_n\vert)(x) \not\to 0$ as $n$ tends to infinity.
	Then there exists $\eps>0$ and a subsequence
	$(f_{n_k})$ such that $(T(t)\vert f_{n_k} \vert)(x) > \eps $ for all $k\in\N$. By assumption (\ref{assertion4}),
	this implies there exists an open neighborhood $U$ of $x$ such that
	\[ (T(t)\vert f_{n_k}\vert)(y) \geq \eps \;\text{for all}\; y\in U.\]
	Thus
	\[ 0< \eps\mu(U) \leq \limsup_{k\to\infty} \int_\Omega (T(t)\vert f_{n_k}\vert) \dx\mu 
	=  \limsup_{k\to\infty} \int_\Omega \vert f_{n_k} \vert \dx\mu =0,\]
	a contradiction.

	(\ref{assertion1}) $\Rightarrow$ (\ref{assertion1*}): This is trivial.
\end{proof}

\begin{rem}
	If the semigroup is merely eventually strong Feller, the assertions of Theorem~\ref{thm:ultrafeller}
	remain equivalent if in each assertion ``for all $t>0$'' is replaced by ``for sufficiently large $t>0$''.
\end{rem}

\begin{rem}\label{rem:L1kernels}
	Assertion (\ref{assertionkernel}) of Theorem~\ref{thm:ultrafeller} implies that the extension of the semigroup to $L^1(\Omega,\mu)$
	consists of kernel operators.
	In fact, the representation of Definition~\ref{def:kerneloperator}
	is valid for bounded functions and by the monotone convergence theorem,
	the same representation remains valid almost everywhere for general $f \in L^1(\Omega,\mu)$.
\end{rem}

\begin{rem}\label{rem:equicontorbits}
	If the semigroup is eventually bounded, i.e., $\sup_{t\geq \tau} \Vert T(t)\Vert <\infty$ for some $\tau>0$,
	assertion (\ref{assertion4}) of Theorem~\ref{thm:ultrafeller} implies that
	$\{ T(t)f : t>t_0\}$ is equicontinuous for all $f\in B_b(\Omega)$ and sufficiently large $t_0>0$.
\end{rem}

The following measure theoretic corollary may be interesting in its own right. Given an integral representation of a semigroup
it allows us to find an $(\Omega\times\Omega)$-measurable kernel for which the representation holds pointwise.

\begin{cor}
	Let $\Omega$ be second countable, $\mathcal{T}=(T(t))_{t\geq 0}$ a 
	positive strong Feller semigroup, and $\mu$ a strictly positive, finite, outer regular, $\mathcal{T}$-invariant 
	Borel measure on $\Omega$. If one of the equivalent conditions of Theorem~\ref{thm:ultrafeller} is satisfied,
	then for every $t>0$ there exists an $(\Omega\times\Omega)$-measurable function $\tilde k_t$ such that
	\[ (T(t)f)(x) = \int_\Omega \tilde k_t(x,y)f(y) \dx\mu(y)\]
	holds for all $x\in\Omega$ and $f\in B_b(\Omega)$.
\end{cor}
\begin{proof}
	Let $t>0$ and fix be countable basis $\mathcal{B}$ of $\Omega$ that is closed under finite unions and contains $\Omega$.
	Since $T(t)$ is a kernel operator, there exists a $\mu$-nullset $N\subset \Omega$ such that
	 \[ (T(t)\mathds{1}_B)(x) = \int_\Omega k_{t}(x,y) \mathds{1}_B(y) \dx\mu(y) \quad (x\in\Omega\setminus N,\,B\in\mathcal{B}).\]
	By assertion~(\ref{assertion1}) of Theorem~\ref{thm:ultrafeller}
	 \[ \lim_{n\to\infty} (T(t)\mathds{1}_{A_n})(x)= (T(t)\mathds{1}_A)(x) \]
	 for all $x\in\Omega$ and all sequences $A_n\subset \Omega$ of Borel sets such that $\lim \mathds{1}_{A_n}=\mathds{1}_A$ $\mu$-almost everywhere.

	 Let $O\subset \Omega$ be open and choose $(B_n)\subset \mathcal{B}$ such that $\lim \mathds{1}_{B_n}(x) =\mathds{1}_O(x)$ for all $x\in\Omega$.
	 Then
	 \[ (T(t)\mathds{1}_O)(x) = \lim_{n\to\infty} (T(t)\mathds{1}_{B_n})(x) 
	 = \lim_{n\to\infty} \int_{B_n} k_{t}(x,y) \dx\mu(y) = \int_O k_{t}(x,y) \dx\mu(y) \]
	 holds for all $x\in\Omega\setminus N$ by the choice of $N$ and the dominated convergence theorem.
	 
	Let $A\subset \Omega$ be a Borel set. Since $\mu$ is outer regular, there exists a sequence $O_n\subset \Omega$ of open sets 
	such that $\lim \mathds{1}_{O_n} = \mathds{1}_A$ almost everywhere. Thus, as in the previous argument,
	\[ (T(t)\mathds{1}_A)(x) = \int_\Omega k_{t}(x,y)\mathds{1}_A(y)\dx\mu(y) \quad (x\in \Omega\setminus N).\]
	Since every function of $B_b(\Omega)$ is the uniform limit of a sequence of simple functions, we conclude that
	\[ (T(t) f) (x) = \int_\Omega k_{t}(x,y) f(y) \dx\mu(y)\]
	holds for every $f\in B_b(\Omega)$ and every $x\in\Omega\setminus N$.
	
	By assertion (\ref{assertion2}) of Theorem~\ref{thm:ultrafeller}, for every $x\in \Omega$ there exists $k_{t,x}\in L^1(\Omega,\mu)$
	such that
		\[ (T(t)f)(x) = \int_\Omega k_{t,x} f\dx\mu \quad (f\in B_b(\Omega))).\]
	Thus, the assertion follows with $\tilde k_t(x,y) \coloneqq \mathds{1}_{\Omega\setminus N}(x) k_t(x,y) + \mathds{1}_N(x) k_{t,x}(y)$.
\end{proof}


\section{Main results}
\label{sec:convergence}

This section contains our main result, a version of Doob's theorem. 
We start with two auxiliary lemmas.

\begin{lemma}
	\label{lem:Tirreducible}
	Let $\mathcal{T}=(T(t))_{t\geq 0}$ be a positive strong Feller semigroup and $\mu$ a strictly positive, finite, $\mathcal{T}$-invariant
	Borel measure on $\Omega$.
	Let $e\in B_b(\Omega)$ be a fixed point of $\mathcal{T}$ such that $e(x)>0$ for all $x\in\Omega$.
	Then the extension of the semigroup $\mathcal{T}$ to $L^1(\Omega,\mu)$ is irreducible if and only if
	there is no open and closed set $\emptyset \neq A\subsetneq \Omega$ such that the closed ideal
	\[ \{ f\in B_b(\Omega) : f_{\mid \Omega \setminus A} =0 \} \]
	of $B_b(\Omega)$ is invariant under the action of $\mathcal{T}$.
\end{lemma}
\begin{proof}
	Assume that there is no open and closed set $\emptyset \neq A \subsetneq \Omega$ such that
	\[ \{ f\in B_b(\Omega) : f_{\mid \Omega \setminus A} =0 \} \]
	is invariant under the action of $\mathcal{T}$.
	Let $J \subset L^1(\Omega,\mu)$ be a closed $\mathcal{T}$-invariant ideal.
	Then there exists a Borel set $M\subset \Omega$ such that
	\[ J = L^1(M,\mu) \coloneqq \{ f\in L^1(\Omega,\mu) : f_{\mid \Omega\setminus M} = 0 \text{ almost everywhere}\}. \]
	Fix $\tau>0$. Then $f\coloneqq T(\tau)( e\mathds{1}_M) \leq e$ is a positive function in $C_b(\Omega)$ with
	$f_{|\Omega\setminus M}=0$ almost everywhere, thus $f\leq e\mathds{1}_M$ almost everywhere.
	Since $\mu$ is $\mathcal{T}$-invariant, $\int_\Omega f \dx\mu = \int_\Omega e\mathds{1}_M \dx\mu$
	and hence $f=e\mathds{1}_M$ almost everywhere. Since $f$ and $e$ are continuous
	and since $\mu$ is strictly positive,
	we conclude from $({f}/{e})^2 = \mathds{1}_M = {f}/{e}$ almost everywhere that $(f/e)^2 = f/e$ holds pointwise.
	Hence $f/e=\mathds{1}_A$ for some open and closed set $A\subset \Omega$ with $\mu(A\vartriangle M)=0$.
	
	Assume that $\emptyset\neq A\neq \Omega$. By hypothesis, there exists $t>0$ such that
	$(T(t)\mathds{1}_A)(x)\neq 0$ for some $x\in \Omega\setminus A$.
	In particular, $\mathds{1}_A\in J$ and $T(t)\mathds{1}_A\not\in J$ since $T(t)\mathds{1}_A$ is continuous, $A$ is closed,
	and $\mu$ is strictly positive. This contradicts the choice of $M$.
	Thus $A=\emptyset$ or $A=\Omega$ and hence $L^1(M)=\{0\}$ or $L^1(M)=L^1(\Omega,\mu)$.

	The converse implication of the lemma is obvious.
\end{proof}

\begin{lemma}
	\label{lem:mustrictlypositive}
	Let $\Omega$ be second countable,
	$\mathcal{T}=(T(t))_{t\geq 0}$ a positive strong Feller semigroup, and $\mu\neq 0$ a positive, finite, $\mathcal{T}$-invariant
	Borel measure on $\Omega$.
	We assume that there is no closed set $\emptyset \neq A\subsetneq \Omega$ such that the closed ideal
	\[ \{ f\in B_b(\Omega) : f_{\mid A} =0 \} \]
	of $B_b(\Omega)$ is invariant under the action of $\mathcal{T}$. Then $\mu$ is strictly positive.
\end{lemma}
\begin{proof}
	Let 
	\[ \mathcal{O} \coloneqq \bigcup_{\substack{V\subset\Omega \text{ open}\\\mu(V)=0}} V.\]
	Then $\mathcal{O}$ is open and $\mu(\mathcal{O})=0$ since $\Omega$ is second countable. Let
	\[ J \coloneqq \{ f\in B_b(\Omega) : f_{\mid \Omega\setminus\mathcal{O}} =0\}.\]
	Since $\mu(\Omega\setminus\mathcal{O})=\mu(\Omega)>0$, we have that $\Omega\setminus\mathcal{O}\neq \emptyset$ and $J\neq B_b(\Omega)$.

	Assume that $\mathcal{O}\neq \emptyset$. Then $J\neq \{0\}$ and,
	by assumption, there exists $0\leq f\in J$ and $t>0$ such that $T(t)f\not\in J$, i.e., there exists $x\in\Omega\setminus\mathcal{O}$
	with $(T(t)f)(x)>0$. Hence we can choose an $\eps>0$ and an open neighborhood $U$ of $x$ such that
	$(T(t)f)(y)\geq \eps$ for all $y\in U$. Since
	\[ 0 = \int_\Omega f\dx \mu = \int_\Omega T(t)f \dx\mu \geq \eps \mu(U),\]
	the open set $U$ is a nullset and therefore $U\subset\mathcal{O}$. This contradicts the choice of $x$.
	Hence $\mathcal{O} = \emptyset$, which proves the claim.
\end{proof}

In the following proposition we formulate two sufficient conditions for the extension of the semigroup to $L^1$ to be irreducible.

\begin{prop}
	\label{prop:TL1irred}
	Let $\Omega$ be second countable, $\mathcal{T}=(T(t))_{t\geq 0}$ a positive strong Feller semigroup, and $\mu\neq 0$ a positive, finite,
	and $\mathcal{T}$-invariant Borel measure on $\Omega$. Consider the following assertions.
	\begin{enumerate}[(i)]
		\item\label{item:mensairred} The semigroup is bounded and
		there is no closed set $\emptyset \neq A\subsetneq \Omega$ such that the closed ideal
	$\{ f\in B_b(\Omega) : f_{\mid A} =0 \}$ of $B_b(\Omega)$ is invariant under the action of $\mathcal{T}$.
		\item\label{item:fixedpoint} There exists $e\in\fix(\mathcal{T})$ such that $e(x)>0$ for all $x\in \Omega$, $\Omega$ is
		connected, and $\mu$ is strictly positive.
		\item\label{item:Lpirred} The extension of the semigroup to $L^1(\Omega,\mu)$ is irreducible and 
		$\mu$ is strictly positive.
	\end{enumerate}
	Then (\ref{item:mensairred}) $\Rightarrow$ (\ref{item:Lpirred}) and (\ref{item:fixedpoint}) $\Rightarrow$ (\ref{item:Lpirred}).
\end{prop}
\begin{proof}
	(\ref{item:mensairred}) $\Rightarrow$ (\ref{item:Lpirred}): First note that $\mu$ is strictly positive by Lemma~\ref{lem:mustrictlypositive}.
	Let $J \subset L^1(\Omega,\mu)$ be a closed $\mathcal{T}$-invariant ideal.
	Then there exists a Borel set $B\subset \Omega$ such that
	\[ J = L^1(B,\mu) \coloneqq \{ f\in L^1(\Omega,\mu) : f_{\mid \Omega\setminus B} = 0 \text{ almost everywhere}\}. \]
	For $t>0$ let $A_t \coloneqq \{ T(t)\mathds{1}_B = 0 \}$, which is a closed subset of $\Omega$.
	Denote by $M\coloneqq \sup_{t\geq 0} \Vert T(t)\Vert$ the bound of $\mathcal{T}$. Then
	$T(s)\mathds{1}_B \leq M \mathds{1}_B$ $\mu$-almost everywhere and hence $(T(t+s)\mathds{1}_B)(x) \leq M (T(t)\mathds{1}_B)(x)$ for every
	$x\in \Omega$. This shows that $A_{t+s} \supset A_t$ for all $t,\, s >0$.

	Let $A\coloneqq \cap_{t>0} A_t$, which is also a closed subset of $\Omega$.
	Since $\Omega\setminus A_t$ is non-increasing w.r.t.\ $t$, we obtain that
	\[ \Omega\setminus A = \bigcup_{t>0} (\Omega\setminus A_t) = \bigcup_{n\in\N} (\Omega\setminus A_\frac{1}{n}).\]
	Hence, $\mu((\Omega\setminus A)\setminus B) = \lim_{n\to\infty} \mu((\Omega\setminus A_\frac{1}{n})\setminus B) = 0$
	and consequently $\mathds{1}_{\Omega\setminus A} \leq \mathds{1}_B$ $\mu$-almost everywhere.
	Thus,
	\[ T(t)\mathds{1}_{\Omega\setminus A} \leq T(t)\mathds{1}_B \leq M \mathds{1}_{\Omega\setminus A_t} 
	\leq M \mathds{1}_{\Omega\setminus A} \quad (t>0),\]
	which shows that the closed ideal $\{ f\in B_b(\Omega) : f_{\mid A} = 0\}$ is $\mathcal{T}$-invariant.
	Now assumption (\ref{item:mensairred}) implies that $A=\Omega$ or $A=\emptyset$.
	In the first case, $A_t = \Omega$ for all $t>0$, hence
	\[ \mu(B) = \int_\Omega T(t) \mathds{1}_B \dx\mu = 0\]
	and $J=\{0\}$.  In the second case we conclude from $\mu((\Omega\setminus A)\setminus B)=0$ 
	that $\mu(B)=\mu(\Omega)$ and $J = L^1(\Omega,\mu)$.

	(\ref{item:fixedpoint}) $\Rightarrow$ (\ref{item:Lpirred}): This follows from Lemma~\ref{lem:Tirreducible}.
\end{proof}

The following lemma is well-known and a consequence of Proposition~3.5 and Theorem~3.8 in~\cite{nagel1986}.
Still, we include its short proof in order to be more self-contained.
\begin{lemma}
	\label{lem:fixedspaceL1}
	Let $\mu$ be a positive Borel measure on $\Omega$ and 
	$\mathcal{T}=(T(t))_{t\geq 0}$ a positive and irreducible 
	semigroup on $L^1(\Omega,\mu)$ such that $\mathds{1}\in\fix(\mathcal{T}')$ and $\fix(\mathcal{T})\neq \{0\}$.
	Then there exists $e\in\fix(\mathcal{T})$ with $e(x)>0$ almost everywhere and
	$\fix(\mathcal{T})=\lin\{e\}$.
\end{lemma}
\begin{proof}
Let $h\in\fix(\mathcal{T})$ and $\tau \geq 0$. 
Since $\vert h\vert = \vert T(\tau)h \vert \leq T(\tau)\vert h \vert$
and \[ \int_\Omega (T(\tau)\vert h\vert - \vert h\vert )\dx\mu = \applied{T(\tau)\vert h\vert - \vert h\vert}{\mathds{1}} = 0,\]
we obtain that $\vert h\vert \in \fix(\mathcal{T})$. Thus $\fix(\mathcal{T})$ is a sublattice.
Since for every $e \in \fix(\mathcal{T})$ the closure of the principal ideal 
\[	\{ f\in L^1(\Omega,\mu) : \vert f\vert \leq c e^+ \text{ for some } c>0\} \]
is closed and $\mathcal{T}$-invariant, $e^+(x)=0$ for almost every $x \in \Omega$
or $e^+(x)>0$ for almost every $x \in \Omega$. Thus elements of $\fix(\mathcal{T})$ do not
change sign, which shows that $\fix(\mathcal{T})$ is totally ordered.
Let $0\leq e\in\fix(\mathcal{T})$ with $e\neq 0$. For every $h\in\fix(\mathcal{T})$ 
there exist $c_1,\,c_2 \in \mathds{R}$ such that $c_1 e \leq h \leq c_2 e$.
For
\[ \lambda \coloneqq \inf\{c\in\mathds{R} : h\leq ce\}\in \mathds{R} \]
we obtain $h\leq \lambda e$. Since also $\lambda = \sup\{c \in \mathds{R} : ce \leq h \}$, we have $h = \lambda e$.
\end{proof}


Our main result is based on the following remarkable theorem due to Greiner, see also
a more recent article by Arendt~\cite[Theorem~4.2]{arendt2008}.

\begin{theorem}[\mbox{\cite[Kor.~3.11]{greiner1982}}]
\label{thm:greiner}
	Let $\mu$ be a positive Borel measure on $\Omega$, $1\leq p <\infty$, and $\mathcal{T}=(T(t))_{t\geq 0}$ a positive
	contraction semigroup on $L^p(\Omega,\mu)$ such that
	$T(t_0)$ is a kernel operator for some $t_0>0$. 
	Let $e\in L^p(\Omega,\mu)$ be a fixed point of $\mathcal{T}$ such that $e(x)>0$ almost everywhere.
	Then $\lim_{t\to\infty} T(t)f$ exists in the norm of $L^p(\Omega,\mu)$ for all $f\in L^p(\Omega,\mu)$.
\end{theorem}

We are now in the position to prove our version of Doob's theorem.

\begin{theorem}
	\label{thm:pointwiseconvergence}
	Let $\Omega$ be second countable, let
	$\mathcal{T}=(T(t))_{t\geq 0}$ be a positive and bounded strong Feller semigroup on $B_b(\Omega)$
	such that $T'(t)\delta_x$ is a Borel measure for all $x\in\Omega$ and $t>0$.
	Let $\mu$ be a positive, finite, $\mathcal{T}$-invariant Borel measure on $\Omega$ and
	assume that one of the three assertions of Proposition~\ref{prop:TL1irred} holds.
	If 
	\begin{align}
	\lim_{t\to 0} \int_\Omega \vert T(t)f - f\vert \dx\mu = 0 \quad (f\in B_b(\Omega)), \label{eqn:L1continuous}\end{align}
	then there exists $e\in \fix(\mathcal{T})$ such that $e(x)>0$ $\mu$-almost everywhere, 
	$\fix(\mathcal{T})=\lin\{e\}$, and 
	\[ \lim_{t\to\infty} T(t)f = \int_\Omega f \dx\mu \; e\]
	for all $f\in B_b(\Omega)$ uniformly on compact subsets of $\Omega$.
\end{theorem}
\begin{proof}
	Let $f\in B_b(\Omega)$ and $(t_n)\subset[0,\infty)$ be a sequence with $\lim t_n=\infty$.
	Then $\{ T(t_n)f : n\in\N \}$ is bounded in norm by assumption 
	and equicontinuous by Theorem~\ref{thm:ultrafeller} and Remark~\ref{rem:equicontorbits}.
	Therefore, $T(t_n)f$ converges to some $g\in C_b(\Omega)$ uniformly on compact subsets 
	of $\Omega$ after passing to a subsequence according to the Ascoli-Arzel\`a theorem, see e.g. \cite[Thm IV.19.1c, p. 219]{benedetto2002}.
	
	By Remark~\ref{rem:extension}, the semigroup $\mathcal{T}$ induces a bounded semigroup $\mathcal{T}_p = (T_p(t))_{t\geq 0}$ on 
	$L^p(\Omega,\mu)$ for all $1\leq p<\infty$ where the semigroup $\mathcal{T}_1$ is contractive.
	Moreover, assumption \eqref{eqn:L1continuous} implies that
	$\mathcal{T}_1$ is strongly continuous and we conclude from
	\[ \Vert T(t)f - f\Vert_2 \leq \sqrt{\Vert T(t)f - f\Vert_1}
	\sqrt{\Vert T(t)f-f\Vert_\infty} \quad (f\in L^\infty(\Omega,\mu)),\, t\geq 0),\]
	that so is $\mathcal{T}_2$.
	Hence, $\mathcal{T}_2$ is mean ergodic by \cite[Ex V.4.7]{nagel2000}. Since $\mathds{1}\in\fix(\mathcal{T}_2')$,
	it follows from the mean ergodic theorem \cite[Thm V.4.5]{nagel2000} that $\fix(\mathcal{T}_2)\neq \{0\}$.

	By Lemma~\ref{lem:fixedspaceL1}, there exists $e\in \fix(\mathcal{T}_1)$ with $e(x)>0$ $\mu$-almost everywhere 
	and $\fix(\mathcal{T}_1)=\lin\{e\}$. We may assume that $\Vert e \Vert_{L^1(\Omega,\mu)}=1$.
	Moreover, $\mathcal{T}_1$ consists of kernel operators by Theorem~\ref{thm:ultrafeller} and Remark~\ref{rem:L1kernels}.
	Thus, $\lim_{t\to\infty} T_1(t)h$
	exists for every $h \in L^1(\Omega,\mu)$ by Theorem~\ref{thm:greiner} and we obtain that
	$L^1(\Omega,\mu) = \fix(\mathcal{T}_1) \oplus R$ for
	\[ R \coloneqq \overline \lin \{ h-T_1(t)h : h\in L^1(\Omega,\mu),\, t\geq 0\} \]
	from results of ergodic theory, see~\cite[V.4.4]{nagel2000}.
	Hence the limit in the norm of $L^1(\Omega,\mu)$ is
	\begin{align}
	\lim_{t\to\infty} T_1(t)f = \int_\Omega f \dx\mu \; e.  \label{eqn:greiner}
	\end{align}
	Therefore, the function $g$ equals the right hand side of \eqref{eqn:greiner} almost everywhere whenever $f \in B_b(\Omega)$.
	This implies that $g$ does not depend on the choice of the sequence $(t_n)$ and hence
	\[ \lim_{t\to\infty} T(t)f = \lim_{n\to\infty} T(t_n)f = g \]
	for all $f\in B_b(\Omega)$ uniformly on compact subsets of $\Omega$.
\end{proof}

\begin{rem}
Let us comment on the strong continuity on $L^1(\Omega,\mu)$.
If $\Omega$ is separable and metrizable and if $\mu$ is a $\mathcal{T}$-invariant finite positive measure
for a semigroup $\mathcal{T}$, then by~\cite[Thm 7.1.7]{bogachev2007} the measure $\mu$ is inner and hence outer regular,
and hence by Urysohn's lemma $C_b(\Omega)$ is dense in $L^1(\Omega,\mu)$.
In this case, if $\lim_{t\to 0} T(t)f = f$ in the norm of $L^1(\Omega,\mu)$ for all $f\in C_b(\Omega)$,
then the extension of $\mathcal{T}$ to $L^1(\Omega,\mu)$ is strongly continuous.
This is in particular the case if $\mathcal{T}$ is stochastically continuous, i.e., if
$\lim_{t\to 0} (T(t)f)(x)=f(x)$ for all $f\in C_b(\Omega)$ and $x\in\Omega$.
So in particular for stochastically continuous Markovian semigroups
on Polish spaces, which is the most common framework for Doob's theorem, our continuity assumption is satisfied.
\end{rem}

\begin{rem}
	The assumption that $\mu$ be finite cannot be dropped in Theorem~\ref{thm:pointwiseconvergence}.
	In fact, if we let $\mathcal{T}$ be the heat semigroup on the real line, simple calculations show that
	for all sufficiently large real numbers $a \ge 1$ the function
	\[
		f_a \coloneqq \mathds{1}_{\bigcup_{n=1}^\infty [a^n,2a^n]} \in B_b(\Omega)
	\]
	has the property that
	$\liminf_{t \to \infty} (T(t)f_a)(0) < \limsup_{t \to \infty} (T(t)f_a)(0)$.
\end{rem}


Seidler~\cite[Prop~2.5]{seidler1997} and Stettner~\cite[Thm~1]{stettner1994}
proved that the convergence in Doob's theorem can be strengthened
to strong convergence of the adjoint semigroup with respect to the total variation norm.
This is still true under the assumptions of Theorem~\ref{thm:pointwiseconvergence}
and can be obtained easily using the following lemma.

\begin{lemma}
\label{lem:equiintegrable}
Let $\Omega$ be second countable, $\mathcal{T}=(T(t))_{t\geq 0}$ a positive strong Feller semigroup,
and $\mu$ a strictly positive, finite, $\mathcal{T}$-invariant Borel measure on $\Omega$.
Assume that $\sup_{t\ge\tau} \|T(t)\| < \infty$ for some $\tau >0$, and
that for all $x\in\Omega$ and $t>0$ there exists $k_{t,x}\in L^1(\Omega,\mu)$ such that
\[ (T(t)f)(x) = \int_\Omega k_{t,x} f \dx\mu \quad (f\in B_b(\Omega)).\]
Then the following assertions hold.
\begin{enumerate}[(a)]
	\item \label{item:equi-integrable} The family $\{ k_{t,x} : t\geq 2\tau \}$ is equiintegrable for all $x\in\Omega$, i.e.,
	for all $\eps>0$ there exists $\delta>0$ such that $\vert \int_A k_{t,x} \dx\mu\vert < \eps$ for every $t\geq 2\tau$ and 
	each Borel set $A\subset \Omega$ with $\mu(A)<\delta$.
	\item \label{item:jointlyconvergence} Let $(g_n)\subset B_b(\Omega)$ be a bounded sequence with $\lim g_n(x) = 0$ for 
	almost every $x\in\Omega$. Then
	$\lim (T(t_n)g_n)(x) = 0$ for all $x\in\Omega$ and every sequence $(t_n) \subset [2\tau,\infty)$.
\end{enumerate}
\end{lemma}
\begin{proof}
	(\ref{item:equi-integrable}) First note that $k_{t,x} \geq 0$ almost everywhere since $\mathcal{T}$ is positive.
	Let $x\in\Omega$. Assume that there exists $\eps_0>0$ and sequences $(t_n) \subset [2\tau,\infty)$
	and $(A_n) \subset\Omega$ such that $\lim \mu(A_n) =0$ and 
	\[ (T(t_n)\mathds{1}_{A_n})(x) = \int_{A_n} k_{t_n,x} \dx\mu \geq \eps_0 \quad (n\in\N).\]
	Since
	\[ \{ T(t_n)\mathds{1}_{A_n} : n\in\N \} 
	\subset \{ T(\tau)f : f\in B_b(\Omega), \, \vert f\vert \leq \sup\nolimits_{t\geq \tau} \Vert T(t)\Vert  \} \]
	is equicontinuous at $x\in\Omega$ by Theorem~\ref{thm:ultrafeller},
	there exists an open neighborhood $U$ of $x$ such that $(T(t_n)\mathds{1}_{A_n})(y) \geq \frac{\eps_0}{2}$ for all $y\in U$.
	Consequently, 
	\[ \mu(A_n) = \int_\Omega \mathds{1}_{A_n} \dx\mu = \int_\Omega T(t_n)\mathds{1}_{A_n} \dx\mu \geq \int_U \frac{\eps_0}{2} \dx\mu 
	\geq \mu(U)\frac{\eps_0}{2}, \]
	contradicting our assumption.

	(\ref{item:jointlyconvergence}) Let $\eps>0$, $x\in\Omega$ and $(t_n)\subset [2\tau,\infty)$.
	By Egoroff's theorem~\cite[Thm~10.38]{aliprantis2006} and part (a) there exists a Borel set $A\subset \Omega$ such that
	$\int_A k_{t_n,x} \dx\mu < \eps$ for all $n\in\N$ and $\lim g_n = 0$ uniformly on $\Omega\setminus A$.
	It follows that
	\begin{align*}
		\vert (T(t_n)g_n)(x) \vert &=  \biggr\vert \int_A k_{t_n,x} g_n \dx\mu + \int_{\Omega\setminus A} k_{t_n,x} g_n \dx\mu \biggr\vert \\
&\leq \Vert g_n \Vert_\infty \int_A k_{t_n,x} \dx\mu + \eps \int_{\Omega\setminus A} k_{t_n,x} \dx\mu \\
&\leq \Vert g_n \Vert_\infty \eps + \eps \sup_{t\geq 2\tau}\Vert T(t)\Vert 
	\end{align*}
	for $n\in\N$ sufficiently large such that $\vert g_n(x) \vert < \eps$ for all $x\in\Omega\setminus A$.\qedhere
\end{proof}

\begin{theorem}\label{thm:adjointconv}
	Under the assumptions of Theorem~\ref{thm:pointwiseconvergence} 
	there exists $e\in \fix(\mathcal{T})$ with $e(x)>0$ $\mu$-almost everywhere, $\fix(\mathcal{T})=\lin\{e\}$, and
	$\Vert e\Vert_{L^1(\Omega,\mu)}=1$ such that
	\[ \lim_{t\to\infty} T'(t) \nu  = \applied{e}{\nu} \mu \]
	for every finite Borel measure $\nu$
	in the norm of $B_b(\Omega)'$, which is equivalent to the total variation norm (see e.g. \cite[Thm 14.4]{aliprantis2006}).
	In particular, up to scalar multiples
	there is only one positive, finite, $\mathcal{T}$-invariant Borel measure $\mu$.
\end{theorem}
\begin{proof}
	The invariant measure $\mu$ is strictly positive by Proposition~\ref{prop:TL1irred}.
	Moreover, Theorem~\ref{thm:pointwiseconvergence} guarantees the existence of
	$e\in \fix(\mathcal{T})$ with $e(x)>0$ almost everywhere and $\fix(\mathcal{T})=\lin\{e\}$
	such that
	\begin{align}
	\lim_{t\to\infty} \applied{ f}{T'(t)\delta_x} = \applied{e}{\delta_x} \applied{f}{\mu}\label{eqn:doobassertion}
	\end{align}
	for all $f\in B_b(\Omega)$ and all $x\in\Omega$.

	Let $\nu$ be a finite Borel measure on $\Omega$, without loss of generality positive, and let $\tau>0$.
	Fix sequences $(t_n)\subset[3\tau,\infty)$
	and $(f_n)\subset B_b(\Omega)$ such that $\lim t_n = \infty$ and $\Vert f_n \Vert_\infty \leq 1$.
	Since the bounded family $\{ T(\tau )f_n : n\in\N \}$ is equicontinuous by Theorem~\ref{thm:ultrafeller},
	the Ascoli-Arzel\`a theorem yields that
	there is $g\in C_b(\Omega)$ such that $\lim (T(\tau)f_n)(x) = g(x)$ for all $x\in\Omega$ after passing to a subsequence,
	see e.g. \cite[Thm IV.19.1c, p. 219]{benedetto2002}.
	Hence, $\lim T(t_n-\tau)(T(\tau)f_n-g)=0$ pointwise by Lemma~\ref{lem:equiintegrable}
	and $\lim T(t_n -\tau )g = \applied{g}{\mu} e$ pointwise by \eqref{eqn:doobassertion}.

	The dominated convergence theorem now yields that
	\begin{align*}
		& \vert \applied{f_n}{T'(t_n)\nu} - \applied{f_n}{\mu}\applied{e}{\nu} \vert  \\
		 & \quad \leq  \int_\Omega \biggr\vert T(t_n)f_n - \int_\Omega f_n \dx\mu \, e \biggr\vert \dx\nu \\
		& \quad \leq \int_\Omega \biggr\vert T(t_n-\tau)(T(\tau)f_n - g)\biggr\vert \dx\nu 
		+ \int_\Omega \biggr\vert T(t_n-\tau)g - \int_\Omega g \dx\mu \, e \biggr\vert \dx\nu  \\
		& \qquad + \int_\Omega \biggr\vert \int_\Omega g\dx\mu - \int_\Omega T(\tau)f_n \dx\mu \biggr\vert \, e \dx\nu \, \to 0 \quad (n\to\infty).
	\end{align*}
	This implies that 
	\[ \lim_{t\to\infty} \applied{f}{T'(t)\nu} = \lim_{t\to\infty} \int_\Omega T(t)f\dx\nu 
	= \int_\Omega f \dx\mu \int e \dx\nu = \applied{f}{\mu} \applied{e}{\nu} \]
	uniformly in $f\in \{ h\in B_b(\Omega) : \Vert h\Vert_\infty \leq 1\}$.
\end{proof}

\section{Application}
\label{sec:examples}
In this section we apply our version of Doob's theorem to semigroups generated by elliptic operators with Neumann 
boundary conditions on a domain $\Omega\subset\mathds{R}^N$ of finite measure.
We do not impose any boundary regularity condition on $\Omega$, but
for simplicity we assume $\Omega$ to be connected.

\smallskip
Let $\Omega\subset\mathds{R}^N$ be open and connected with $\lambda(\Omega) < \infty$, where $\lambda$ denotes
the Lebesgue measure. Let $a_{ij}\in L^\infty(\Omega)$ be real-valued functions. We assume that the sesquilinear
form $a$ given by
\[ a(u,v) \coloneqq \sum_{i,j=1}^N \int_\Omega a_{ij} \mathrm{D}_i u \, \overline{\mathrm{D}_j v} \; \dx \lambda \quad (u,v\in H^1(\Omega)) \]
is sectorial, i.e., that there exists $\theta \in [0,\frac{\pi}{2})$ such that
\[
	\sum_{i,j=1}^N a_{ij} \xi_i \overline{\xi_j} \in \Sigma_\theta \coloneqq \bigl\{ r e^{i\tau} : r \ge 0, \; |\tau| \le \theta \bigr\}
\]
almost everywhere on $\Omega$ for all $\xi \in \mathds{C}^N$.
We also assume that $a$ is locally strongly elliptic, i.e., that 
for every compact set $K \subset \Omega$ there exists $\alpha_K > 0$ such that
\[
	\Real \sum_{i,j=1}^N a_{ij} \xi_i \overline{\xi_j} \geq \alpha_K \vert\xi\vert^2
\]
almost everywhere on $K$ for all $\xi \in \mathds{C}^N$.
We set $\|u\|_a^2 \coloneqq \Real a(u,u) + \|u\|_{L^2}^2$ for $u \in D(a) = H^1(\Omega)$.

The form $a$ is associated with an m-sectorial operator $-A$ on $L^2(\Omega)$ by \cite[Lem 4.1 and Thm 3.2]{AreEls08},
and more precisely $A$ generates a real, positive, analytic $C_0$-semigroup $\mathcal{T}=(T(t))_{t\geq 0}$
on $L^2(\Omega)$ with $\mathds{1}$ being in the fixed space of $\mathcal{T}$ as well as the fixed space of
the adjoint semigroup $\mathcal{T}^\ast$.
Moreover, $\mathcal{T}$ extends to a contraction semigroup on $L^p(\Omega)$ for every $p \in [1,\infty]$, which is strongly
continuous for $p \in [1,\infty)$, see \cite[Prop 4.8 and Cor 4.9]{AreEls08}.

Given a measurable set $\omega \subset \Omega$ with $0 < \lambda(\omega) < \lambda(\Omega)$,
there exists a function $u \in H^1(\Omega)$ such that $u \mathds{1}_\omega \not\in H^1_{\mathrm{loc}}(\Omega)$.
Since every Cauchy sequence of $(H^1(\Omega),\|\cdot\|_a)$ converges in $H^1_{\mathrm{loc}}(\Omega)$, the function
$u \mathds{1}_\omega$ is not the limit in $L^2(\Omega)$ of a Cauchy sequence in
$(H^1(\Omega),\|\cdot\|_a)$. Hence we obtain from \cite[Prop 3.11]{AreEls08} that
$\mathcal{T}$ is irreducible on $L^2(\Omega)$, compare \cite[Thm 2.7]{ouhabaz2005},
and thus also on $L^1(\Omega)$,

By the De Giorgi-Nash theorem and a standard bootstrapping argument, $D(A^n)$ consists of locally H\"older continuous functions
if $n\in\N$ is sufficiently large. Since $\mathcal{T}$ is analytic, this implies that $T(t)f \in D(A^\infty) \subset C(\Omega)$ for 
all $f\in L^2(\Omega)$ and $t>0$. In particular, $T(t)f \in C_b(\Omega)$ whenever $f\in L^\infty(\Omega)$.
Thus we can identify $\mathcal{T}$ with a positive, bounded, strong Feller semigroup on
$B_b(\Omega)$ with invariant measure $\lambda$.
Finally, since $T(t)L^2(\Omega) \subset L^\infty_{\mathrm{loc}}(\Omega)$ for all $t>0$, we deduce from
\cite[Prop 1.7]{arendt1994} that $\mathcal{T}$ consists of kernel operators and thus
satisfies the equivalent conditions of Theorem~\ref{thm:ultrafeller}.

From the above we conclude on the basis of Theorem~\ref{thm:pointwiseconvergence} that
\[ \lim_{t\to\infty} (T(t)f)(x) = \lambda(\Omega)^{-1} \int_\Omega f \dx\lambda\]
for all $f\in B_b(\Omega)$ and all $x\in \Omega$.

\smallskip

There is also a more direct way to 
obtain pointwise convergence in the above situation. By Theorem~\ref{thm:greiner}, the semigroup
$(T(t))_{t\geq 0}$ on $L^1(\Omega)$ converges strongly to its mean ergodic projection $P$ and it follows from the 
closed graph theorem that $T(1)$ is a continuous mapping from $L^1(\Omega)$ to $C(\Omega)$ endowed with the compact--open topology.
This implies that
\[ \lim_{t\to\infty} T(t)f = \lim_{t\to\infty} T(1)T(t-1)f = T(1)Pf = Pf \quad (f\in L^1(\Omega)) \]
in $C(\Omega)$, i.e., uniformly on compact subsets of $\Omega$.

\section*{Acknowledgment}
The first author was supported by the graduate school 
\emph{Mathematical Analysis of Evolution, Information and Complexity} during the work on this article.

\bibliographystyle{amsalpha}

\bibliography{newdoob}

\end{document}